\documentclass[12pt]{amsart}

\usepackage{amssymb}

\numberwithin{equation}{section} 
\newtheorem{theorem}{Theorem}[section]
\newtheorem{lemma}[theorem]{Lemma}
\newtheorem{proposition}[theorem]{Proposition}
\newtheorem{corollary}[theorem]{Corollary}
\theoremstyle{definition}
\newtheorem{definition}[theorem]{Definition}

\newtheorem{example}[theorem]{Example}

\newtheorem{remark}[theorem]{Remark}

\newcommand{\cD}{\mathcal{D}}
\newcommand{\cbD}{\mathbf{D}}
\newcommand{\cE}{\mathcal{E}}
\newcommand{\cS}{\mathcal{S}}

\newcommand{\cA}{\mathcal{A}}
\newcommand{\cY}{\mathcal{Y}}

\newcommand{\h}{\mathfrak{h}}
\newcommand{\bZ}{\mathbb{Z}}

\newcommand{\bC}{\mathbb{C}}
\newcommand{\Hom}{\mathrm{Hom}}
\newcommand{\Ext}{\mathrm{Ext}}
\newcommand{\gr}{\mathrm{gr}}
\newcommand{\GL}{\mathrm{GL}}
\newcommand{\codim}{\mathrm{codim}}

\newcommand{\e}{\mathbf{e}}
\newcommand{\Pic}{\mathrm{Pic}}

\newcommand{\cL}{\mathcal{L}}
\newcommand{\cO}{\mathcal{O}}
\newcommand{\reg}{\mathrm{reg}}
\newcommand{\rH}{\mathrm{H}}

\newcommand{\bm}{\mathbf{m}}
\newcommand{\bp}{\mathbf{p}}
\newcommand{\bq}{\mathbf{q}}
\newcommand{\bx}{\mathbf{x}}

\newcommand{\bL}{\mathbf{L}}

\begin{document}
\title{On elliptic Calogero-Moser systems for complex
crystallographic reflection groups}

\author[P.~Etingof]{Pavel Etingof}
\address{{\bf P.E.}: Department of Mathematics, Massachusetts Institute of Technology,
Cambridge, MA 02139, USA}
\email{etingof@math.mit.edu}

\author[G.~Felder]{Giovanni Felder}
\address{{\bf G.F.}: Department of mathematics, ETH Z\"urich, 8092 Z\"urich, Switzerland}
\email{felder@math.ethz.ch}

\author[X.~Ma]{Xiaoguang Ma}
\address{{\bf X.M.}: Department of Mathematics, Massachusetts Institute of Technology,
Cambridge, MA 02139, USA}
\email{xma@math.mit.edu}

\author[A.~Veselov]{Alexander Veselov}
\address{{\bf A.V.}: Department of Mathematical Sciences, Loughborough University, Loughborough LE11 3TU, 
UK and Department of Mathematics and Mechanics, Moscow State University, Moscow, 119899, Russia}
\email{A.P.Veselov@lboro.ac.uk}

\begin{abstract}
To every irreducible finite crystallographic reflection group
(i.e., an irreducible finite reflection group $G$ acting faithfully on an abelian variety $X$),
we attach a family of classical and quantum integrable systems on $X$ (with meromorphic 
coefficients). These families are parametrized by $G$-invariant functions of pairs $(T,s)$,
where $T$ is a hypertorus in $X$ (of codimension 1), 
and $s\in G$ is a reflection acting trivially on $T$. 
If $G$ is a real reflection group, these families reduce to the 
known generalizations of elliptic Calogero-Moser systems, but in the non-real case they 
appear to be new. We give two constructions of the integrals of these systems -
an explicit construction as limits of classical Calogero-Moser Hamiltonians 
of elliptic Dunkl operators as the dynamical parameter goes to $0$ (implementing 
an idea of \cite{BFV}), and a geometric construction as global sections of sheaves of elliptic Cherednik 
algebras for the critical value of the twisting parameter. We also prove algebraic integrability of these 
systems for values of parameters satisfying certain integrality conditions.     
\end{abstract}

\maketitle

\vskip .1in
\centerline{\bf To Corrado De Concini on his 60-th birthday with admiration}
\vskip .1in

\section{Introduction}

Classical and quantum integrable many-particle systems on the line 
have been a hot topic since 1970s. Among these, especial attention has
been paid to Calogero-Moser systems with rational, trigonometric, and
elliptic potential. A review of the history of these systems and a variety
of important applications with references can be found in \cite{CMS}.
In particular, in \cite{OP1,OP2}, Calogero-Moser systems 
were generalized to 
the case of any root system, so that the many-particle systems 
of \cite{C1,C2,S,M} correspond to type $A$. 

There are a number of ways to construct Calogero-Moser systems and to 
prove their integrability. One is the Lax matrix method 
(\cite{M, C1, CMR,K1, OP1, OP2, BCS}).
Another, related method is Hamiltonian reduction 
(\cite{KKS},  \cite{F}, \cite{E1} (see in particular the remark at the 
end of section III), \cite{GN}). 
The third method is based on computing radial parts of Laplace 
operators on symmetric spaces (\cite{BPF},\cite{OP4}); 
this method produces quantum Calogero-Moser systems 
only for some special values of coupling constants. 
The fourth method is based on the analytic study of
hypergeometric functions associated to root systems and is
due to Heckman and Opdam (\cite{HO, He3, O1,O2}); it yielded 
the first proof of integrability of rational 
and trigonometric Calogero-Moser systems for any
root system. Finally, the fifth method, most
relevant to this paper, is due to G. Heckman \cite{He1, He2}, 
and is based on considering invariant polynomials of
Dunkl operators \cite{D}. Namely, 
Heckman managed to use this method to give a simple algebraic
proof of the integrability of Calogero-Moser systems 
for any root system and any values of coupling constants
in the rational and trigonometric cases; his method was further
improved by Cherednik \cite{Ch1}. Later Cherednik \cite{Ch2} settled the 
elliptic case, by introducing Dunkl operators for affine root systems. 

An alternative approach to proving the integrability of the quantum elliptic Calogero-Moser system 
was proposed in the paper \cite{BFV}, which introduces the elliptic counterparts of Dunkl operators. 
However, this approach did not quite succeed, because of the following difficulty: 
elliptic Dunkl operators depend on a ``dynamical'' parameter $\lambda$ (lying in the reflection representation 
of the Weyl group $W$), and are not $W$-invariant, but rather $W$-equivariant (i.e., $\lambda$ 
is also transformed by $W$); so to get $W$-invariant Hamiltonians from invariant polynomials of elliptic Dunkl operators, 
one would have to set $\lambda$ to $0$, which is impossible since the elliptic Dunkl operators have poles on the root hyperplanes.  
It was suggested in \cite{BFV} that these poles should be cancelled by some kind of a subtraction procedure
(namely, the calculation in the $A_2$ case made in \cite{BFV}, p. 909, 
indicated that the classical integrals in the dynamical variables may 
be used here), but it was unclear what exactly this procedure should be. 

Since the paper \cite{BFV}, it has seemed certain 
to the authors that elliptic Dunkl operators are ``the right'' objects.
For example, in \cite{EM1}, they were generalized to the case of finite crystallographic complex reflection groups
(following the generalization of usual Dunkl operators to finite complex reflection groups in \cite{DO}), 
and linked to double affine Hecke algebras and Cherednik algebras on complex tori. 
Yet, the problem of providing a precise connection between elliptic Dunkl operators and elliptic Calogero-Moser systems 
remained open. 

The goal of this paper is to finally solve this problem, and   
use the approach of \cite{BFV} to give a new proof of the integrability of 
quantum elliptic Calogero-Moser systems. In fact, 
our main result is more general: we use the  
elliptic Dunkl operators of \cite{EM1} to attach a family of classical and quantum integrable
systems to every finite irreducible crystallographic complex reflection group
$G$, i.e. a finite irreducible complex reflection group acting faithfully on a 
complex torus (preserving $0$)
\footnote{Such groups were classified by Popov [Po] 
(see also \cite{Ma}).}.
When $G$ is a real reflection group (i.e., a Weyl group), 
our construction reproduces the elliptic Calogero-Moser system 
attached to $G$ (in fact, in the $BC_n$ case it reproduces the
full 5-parameter Inozemtsev system \cite{I}). On the other hand, 
when $G$ is not real, we obtain new examples of 
integrable systems with elliptic coefficients. 
We will call these systems {\it crystallographic elliptic
Calogero-Moser systems.}\footnote{We note that these new integrable systems 
may not have a direct physical meaning, since their Hamiltonians are polynomials in momenta 
of degree higher than $2$ and in general have complex coefficients.} 
The simplest example is given by (\ref{cubiccase}) below.

The main idea of our construction is to consider the {\it classical} 
Calogero-Moser Hamiltonians (in the rational case constructed by Heckman's method \cite{He2} as $G$-invariant 
polynomials of the usual Dunkl operators \cite{D}), and substitute the elliptic Dunkl operators
for momentum variables, and the dynamical parameters $\lambda$ for the position variables. 
Our main result is that the resulting operators are regular in $\lambda$
near $\lambda=0$ (i.e., this construction provides the cancelation of poles asked for in \cite{BFV}).
Thus we can now set $\lambda=0$ and obtain a collection of $G$-invariant commuting operators. 
If we restrict these operators to the space of $G$-invariant functions, they become differential operators, 
and thus yield the desired integrable system. 

We also give a geometric construction of crystallographic
elliptic Calogero-Moser systems,  
as global sections of sheaves of elliptic Cherednik 
algebras for the critical value of the twisting parameter.
This is a construction in the style of the Beilinson-Drinfeld
construction of the quantum Hitchin system, \cite{BD}, as global sections of the 
sheaf of twisted differential operators on the moduli stack 
of principal bundles over a curve, for critical twisting.

Finally, we establish algebraic integrability of the quantum crystallographic elliptic Calogero-Moser systems for parameters satisfying certain integrality conditions. 

The paper is organized as follows. 
In Section 2, we recall the basics on complex reflection groups,  
rational Dunkl operators, rational Calogero-Moser Hamiltonians, and 
elliptic Dunkl operators. In Section 3, we state the main theorem
and give some examples. In Section 4, we describe the main new example, 
attached to the groups $S_n\ltimes (\Bbb Z/m\Bbb Z)^n$, where $m=3,4$ or $6$. 
In Section 5, we give two different proofs of the main theorem,
and explain the relation between our arguments and the ones of \cite{Ch2}.
In Section 6, we give the geometric construction of the crystallographic elliptic
Calogero-Moser systems. In Section 7, we establish their algebraic quantum integrability
of our quantum integrable systems at integer points.

{\bf Acknowledgements.} 
The work of P.E. and X.M. was  partially supported by the NSF grants
DMS-0504847 and DMS-0854764. The work of G.F. was partially supported by SNF grant 200020-122126. 
The authors are grateful to O. Chalykh and E. Rains for useful discussions.

\section{Preliminaries}

\subsection{Complex tori and line bundles on them}\label{tori}

Let $\h$ be a finite dimensional complex vector space, and 
$\h^\vee$ be the Hermitian dual of $\h$ (i.e. the dual $\h^*$ with the conjugate complex structure).
Suppose that $\Gamma\subset \h$ is a cocompact lattice. Then $X=\h/\Gamma$ is a complex torus. 

Let $X^{\vee}$ be the dual torus to $X$, i.e.
$X^\vee=\h^\vee/\Gamma^\vee,$
where $\Gamma^{\vee}$ is the dual lattice to $\Gamma$
under the form ${\rm Im}\lambda(v)$, $v\in \h$, 
$\lambda\in \h^\vee$.

It is well known that $X^\vee$ is naturally identified with $\Pic_{0}(X)$, 
the set of classes of topologically trivial holomorphic line bundles on $X$.  
For any $\lambda\in \h^{\vee}$, 
let $\cL_\lambda\in \Pic_{0}(X)$ denote the
corresponding holomorphic line bundle; it is obtained by 
taking the quotient of the trivial line bundle on $\h$ by the $\Gamma$-action 
given by the formula
$$
\gamma(\bx, z)=(\bx +\gamma,e^{2\pi {\rm i}{\rm Im}\lambda(\gamma)}z),
\gamma\in \Gamma.
$$ 
Note that the line bundle $\cL_\lambda$ comes with a natural
Hermitian structure and flat unitary connection (coming from the constant 
ones on $\h$). We will denote this connection by $\nabla$. 

If $X$ is 1-dimensional (an elliptic curve), then we have a
natural identification $X\cong X^\vee$, sending $\bx\in X$ to the
bundle $\cO(\bx)\otimes \cO(0)^*$. This identification yields a natural 
positive Hermitian form $\langle~,~\rangle$ on the line $T_0X^\vee$.  
Hence, for every hypertorus $T\subset X$ passing through $0$ (of
codimension $1$), there is a natural positive Hermitian form 
$\langle~,~\rangle$ on the line $T_0(X/T)^\vee=T_0((X/T)^\vee)$. 

\subsection{Complex reflection groups} 

Let $\h$ be a finite dimensional complex vector space. 
A semisimple element $s\in \GL(\h)$ is called a complex reflection if 
${\rm Im}(1-s)$ is 1-dimensional. For a complex reflection $s$, 
let $\zeta_s=\det(s|_{\h^*})$,  
and let $\alpha_s$ be a nonzero 
linear function on $\h$ vanishing on the fixed hyperplane of $s$.

Let $\cS$ be the set of complex reflections in $G$.
For any $s\in \cS$, we have a decomposition:
\begin{equation*}\label{decomp}
\h=\h^{s}\oplus  {\h}_{s},
\end{equation*}
where $\h^{s}$ is the codimension 1 subspace of $\h$ with
the trivial action of $s$, and 
${\h}_{s}=((\h^{*})^{s})^{\perp}$, 
which is an $s$-invariant $1$-dimensional space. We also have a similar
decomposition on the dual space: 
$\h^{*}=(\h^{*})^{s}\oplus  {\h}_{s}^{*}$.

A finite subgroup 
$G\subset GL(\h)$ is called a {\it complex reflection group}
if it is generated by complex reflections. The representation $\h$
is then called the reflection representation of $G$. 
A complex reflection group $G$ is called irreducible if $\h$ is an irreducible representation 
of $G$. 

Let $G$ be a complex reflection group with reflection
representation $\h$. For a hyperplane $H\subset \h$,
denote by $G_H\subset G$ the stabilizer of a generic
point in $H$. We call $H$ a {\it reflection hyperplane} if $G_H$ is nontrivial;
in this case, $G_H$ is a cyclic group. Let $\h_{\reg}$ be the complement 
of the reflection hyperplanes in $\h$.

By the Shephard-Todd-Chevalley theorem (see \cite{Che}), 
if $G$ is a complex reflection group, then $(S\h)^G$
is a polynomial algebra. Let 
$P_i, i=1, \ldots, n$, denote homogeneous generators of $(S\h)^G$. 

\subsection{Complex tori with an action of a complex reflection group}{\label{sec:cplxact}}

Let $G$ be an irreducible complex reflection group with
reflection representation $\h$, and let $\Gamma\subset \h$ be a
cocompact lattice preserved by $G$, i.e., $G$ is a crystallographic complex reflection group. 
Then we get a $G$-action on the complex torus $X=\h/\Gamma$ preserving $0$.

The action of $G$ on $X$ induces a $G$-action on the dual torus
$X^{\vee}\cong \Pic_{0}(X)$. For a line bundle $\cL\in \Pic_0(X)$,
denote the image of $\cL$ under $g$ by $\cL^g$. 

For any complex reflection $s\in G$,
let $X^s$ be the set of $x\in X$
s.t. $sx=x$.  Connected components
of $X^s$ (which all have codimension $1$)
are called {\it reflection hypertori}. 
Among them, there is one passing through $0$, 
which we denote by $T_s$. 
Let $m_s$ be the order of $s$ (note that $m_s=2,3,4$ or $6$), and let $j(s)\in \lbrace{1,\ldots,m_s-1\rbrace}$ 
be such that $\zeta_s=e^{-2\pi {\rm i}j(s)/m_s}$. 

Let $X_{\reg}$ be the
complement of the reflection hypertori in $X$.
For a reflection hypertorus $T\subset X$, denote by $G_T\subset G$
the stabilizer of a generic point in $T$.
It is a cyclic group. Denote its order by $m_T$ (so $m_{T_s}=m_s$), 
and let $s_T$ be the generator of $G_T$ 
which acts on the normal bundle of $T$ by 
multiplication by $e^{2\pi {\rm i}/m_T}$. 

Denote by $\cA$ the set 
$$
\{(T, j)| T \text{ is a reflection hypertorus }, j=1, \ldots, m_{T}-1\}.
$$ 

For any reflection hypertorus $T$, 
$\h_{s_T^j}$ is independent on $j$, so we will denote it by $\h_T$. 

\begin{remark}
Note that the complex torus 
$X$ in our situation is always 
an abelian variety, which is isogenous to 
a product of elliptic curves. Indeed, 
let $s_1,\ldots,s_n\in G$ be a collection of reflections such that
$\lbrace{\alpha_{s_1},\ldots,\alpha_{s_n}\rbrace}$ is a basis of 
$\h^*$. Then the natural map 
$$
X\to X/T_{s_1}\times\cdots\times
X/T_{s_n}
$$ 
is an isogeny. 
\end{remark} 

\subsection{Dunkl operators for complex reflection groups}{\label{sec:dunkl}}

Let us recall the basic theory of Dunkl operators for complex reflection groups 
(see \cite{DO}, \cite{EM2}). 

Let $c: \cS\to \bC$ be a $G$-invariant function.
The (rational) {\it Dunkl operators} for $G$ are 
the following family of pairwise commuting linear operators
acting on the space of rational functions on $\h$:
\begin{equation}{\label{eqn:clado}}
\cbD_{v,c}=\partial_{v}+\sum_{s\in \cS}
\frac{2c(s)\alpha_{s}(v)}{(1-\zeta_s)\alpha_{s}}s, 
\end{equation}
where $v\in \h$, and $\partial_{v}$ is the derivation associated to the vector
$v$.\footnote{This definition of Dunkl operators is slightly different 
from the one in \cite{E3}, \cite{EM2}, namely we have replaced $s-1$ by $s$. 
This has no significant effect on the considerations below, since our Dunkl 
operators are conjugate to the ones in \cite{E3}, \cite{EM2}.}
Thus, the Dunkl operators are elements in $\bC G\ltimes D(\h_{\reg})$, where 
$D(\h_{\reg})$ denotes the algebra of differential operators on $\h_{\reg}$.

Similarly, one defines the quasiclassical limits of Dunkl operators, 
called the {\it classical Dunkl operators}, which are
elements of $\bC G\ltimes \cO(T^*\h_{\reg})$. 
Namely, for $v\in \h$, let $p_v$ be the corresponding momentum coordinate
in $\cO(T^*\h_{\reg})$. Then the classical Dunkl
operators are defined by the formula
\begin{equation*}
\cbD_{v,c}^0=p_{v}+\sum_{s\in \cS}
\frac{2c(s)\alpha_{s}(v)}{(1-\zeta_s)\alpha_{s}}s, 
\end{equation*}
which is obtained by replacing the derivative $\partial_{v}$ by its symbol $p_{v}$
in \eqref{eqn:clado}.

\subsection{Calogero-Moser Hamiltonians}\label{ccmh}

Let ${\bf m}: \bC G\ltimes D(\h_{\reg})\to D(\h_{\reg})$
be the map defined by the 
formula ${\bf m}(Lg)=L$, where $L\in D(\h_{\reg})$.  
Define the $G$-invariant differential operators $\widehat P_i^c$ on $\h_{\reg}$ by the formula 
$$
\widehat P_i^c:={\bf m}(P_i(\cbD_{\bullet,c})).
$$
In other words, $\cbD_{\bullet,c}$ is a linear map $\h\to \bC  G\ltimes D(\h_{\reg})$ whose image is commuting, 
so it defines an algebra homomorphism $S\h\to \bC  G\ltimes
D(\h_{\reg})$, and ${P}_i(\cbD_{\bullet,c})$ 
is the image of $P_i$ under this homomorphism. 
Note that $\widehat P_i^0=P_i(\partial)$. It is known (see \cite{He2}, \cite{EM2}, \cite{BC}) that these operators are pairwise 
commuting (i.e., form a quantum integrable system).
They are called the {\it rational Calogero-Moser operators}.

Similarly, one can define the quasiclassical limits of 
$\widehat P_i^c$. Namely, let ${\bf m}: \bC G\ltimes \cO(T^*\h_{\reg})\to \cO(T^*\h_{\reg})$
be the map defined by the formula ${\bf m}(Pg)=P$, where $P\in \cO(T^*\h_{\reg})$.  
Define the $G$-invariant functions $P_i^c\in \cO(T^*\h_{\reg})$ by the formula 
$$
P_i^c(\bp,\bq):={\bf m}(P_i(\cbD_{\bullet,c}^0)).
$$
(Here $\bq\in \h$ is the position variable, and $\bp\in \h^*$ is the momentum variable).
Note that $P_i^0=P_i(\bp)$.
It is known (see \cite{He2}, \cite{EM2}) that these functions are pairwise 
Poisson commuting (i.e., form a classical integrable system). 
They are called the {\it rational classical Calogero-Moser Hamiltonians}.

The following important lemma will be used below. 

\begin{lemma}{\label{lem:CCM}}
$P_i(\cbD_{\bullet,c})$ is a function on $T^*\h_{\reg}$,  
i.e., it does not involve nontrivial elements of $G$.
Thus, $P_i^c=P_i(\cbD_{\bullet,c}^0)$, i.e. the application of ${\bf m}$ 
is not necessary. 
\end{lemma}

Note that this lemma does not hold in the quantum setting. 

\begin{proof}
Consider the classical rational Cherednik algebra for $G$,
$H_{0,c}(G,\h)$, generated inside $\bC  G\ltimes \cO(T^*\h_{\reg})$ 
by $G$, $S\h^*$ (the algebra of polynomials on $\h$), and the classical Dunkl operators
(see \cite{E3}, Section 7, and \cite{EM2}, Section 3). 

It is easy to see that $(S\h^*)^G$ 
is contained in the center of $H_{0,c}(G,\h)$.
On the other hand, there is an isomorphism 
$H_{0,c}(G,\h^*)\to H_{0,c}(G,\h)$
which maps linear functions on $\h^*$ to classical Dunkl operators on $\h$
(see \cite{EM2}, proof of Prop. 3.16). 
Thus, for any $P\in (S\h)^G$, $P(\cbD_{\bullet,c}^0)$ is also in the center 
of $H_{0,c}(G,\h)$, and thus, in the center of $\bC  G\ltimes \cO(T^*\h_{\reg})$. 
So $P(\cbD_{\bullet,c}^0)$ commutes with functions of $\bp$ and $\bq$, 
and hence is itself a function.
\end{proof}

\subsection{Elliptic Dunkl operators}{\label{sec:EDO}}

Let $G,X$ be as above. 
Fix a generic line bundle $\cL\in \Pic_0(X)$ (i.e., such that 
$\cL^{g}\ne \cL$ for any reflection $g$). From \cite{EM1}, we know that 
for any $(T, j)\in \cA$, there is a  unique global meromorphic 
section $f_{T,j}^{\cL}$ of the bundle
$(\cL^{s_{T}^{j}})^{*}\otimes \cL\otimes \h_{T}^{*}$
which has a simple pole along $T$ with residue $1$ and no other
singularities.

Let $C$ be a $G$-invariant function on $\cA$.   
Recall from \cite{EM1} that the {\it elliptic Dunkl operator} corresponding to 
$\cL,\nabla,C$, and a vector $v\in \h$ 
is the following operator acting on the local meromorphic 
sections of $\cL$:
\begin{equation*}
\cD_{v,C}^{\cL}=\nabla_v+\sum_{(T, j)\in \cA}C(T, j) 
(f_{T, j}^{\cL},v) s_{T}^{j}.
\end{equation*}
(Here we regard $\h_T^*$ as a subspace of $\h^*$ in a natural
way, using that $\h_T$ has a distinguished complement in $\h$). 

\begin{example}
In the Weyl group case the elliptic Dunkl operators are the operators from \cite{BFV}:
\begin{equation*}
\cD_{v,C}^{\lambda}=\nabla_v+\sum_{\alpha\in R_+}C_{\alpha} (\alpha,v) \sigma_{(\alpha^{\vee},\lambda)}(\alpha) s_{\alpha},
\end{equation*}
where
$$
\sigma_\mu(z)=\frac{\theta(z-z_0-\mu)\theta'(0)}{\theta(z-z_0)\theta(-\mu)},
$$
and 
$$\theta(z)=\theta_1(z,\tau)=-\sum_{n=-\infty}^{\infty} e^{2\pi {\rm i} (z+1/2)(n+1/2)+\pi {\rm i} \tau(n+1/2)^2}$$
 is the first Jacobi theta-function \cite{WW}. 
\end{example}

\begin{remark}
This differs from the definition of  \cite{EM1}
by the sign of $C$. We choose this sign convention 
to reconcile the notation with texts on rational Cherednik
algebras, e.g. \cite{E3} and \cite{EM2}.  
\end{remark}

\begin{proposition}\label{prope} (\cite{BFV, EM1})
The elliptic Dunkl operators have the following properties.
\begin{enumerate}
\item[]
\item commutativity:
$[\cD_{v,C}^{\cL}, \cD_{v',C}^{\cL}]=0$, 
for any $v,v'\in \h$.
\item equivariance: $g\circ \cD_{v,C}^{\cL}\circ g^{-1}=
\cD_{gv,C}^{\cL^g}$, where $g\in G$.
\end{enumerate}
\end{proposition}

It is also useful to consider {\it classical elliptic Dunkl operators}, which are quasiclassical limits of elliptic
Dunkl operators. These operators are parametrized by $v,C,\cL$, and are given by the formula
\begin{equation*}
\cD_{v,C}^{0,\cL}=p_v+\sum_{(T, j)\in \cA}C(T, j)
(f_{T, j}^{\cL},v) s_{T}^{j}.
\end{equation*}
The properties of these operators are similar to those of the quantum elliptic 
Dunkl operators. 

\subsection{Behavior of elliptic Dunkl operators near $0\in \Pic_0(X)$.}

Let $\alpha_T:=\alpha_{s_T}\in \h_T^*$.
Then we have 
$$
f_{T,j}^{\cL_\lambda}=\varphi_{T,j}(\lambda)\alpha_T, 
$$
where $\varphi_{T,j}(\lambda)$ is a section of 
$(\cL_\lambda^{s_{T}^{j}})^{*}\otimes \cL_\lambda$. 
We are going to study the behavior of this section near $\lambda=0$. 

Fix a $G$-invariant positive definite Hermitian form\footnote{We agree 
that Hermitian forms are linear in the first argument 
and antilinear in the second one.} 
$B(~,~)$ on $\h^\vee$ (which is unique up to a positive factor), 
and use it to identify $\h$ with $\h^\vee$; 
so the element of $\h$ corresponding to $\lambda\in \h^\vee$ 
will be denoted by $B(\lambda)$.

For a reflection $s\in G$, set 
$$
a_B(s)=\frac{B(u,u)}{\langle u,u\rangle}
$$
for $0\ne u\in T_0(X/T_s)^\vee$
(where $\langle~,~\rangle$ is defined in Subsection \ref{tori}). 

\begin{proposition}\label{asy}
The section $\widetilde\varphi_{T,j}(\lambda):=
B(\lambda,\alpha_T)\varphi_{T,j}(\lambda)$ is regular in $\lambda$ near $\lambda=0$, 
and if $B(\lambda,\alpha_T)=0$ (i.e., $s_T\lambda=\lambda$), we have  
$$
\widetilde\varphi_{T,j}(\lambda)=-\frac{a_B(s_T)}{1-e^{2\pi {\rm i}j/m_T}}.
$$
\end{proposition} 

\begin{proof}
Suppose $E=\bC /(\bZ\oplus \bZ\tau)$, 
$\mu\in E$, and $\mathcal E=\cO(\mu)\otimes \cO(0)^*$
is a degree zero holomorphic line bundle on $E$. 
Let $\sigma_\mu$ be a section of $\mathcal E$ 
with a first order pole at a point $z_0$ and no other singularities.
Then, up to scaling, we have 
$$
\sigma_\mu(z)=\frac{\theta(z-z_0-\mu)\theta'(0)}{\theta(z-z_0)\theta(-\mu)},
$$
where, as before, $\theta$ is the first Jacobi theta-function. 
Near $\mu=0$, this function has the expansion 
\begin{equation}\label{sigmaf}
\sigma_\mu(z)=-\frac{1}{\mu}+O(1).
\end{equation}

Now let $E=X/T_s$ (an elliptic curve). It is clear that 
the bundle $(\cL_\lambda^{s_T^j})^*\otimes \cL_\lambda$ is pulled back from $E$, namely it is 
the pullback of the line bundle $\mathcal E$ corresponding to 
the point $(1-e^{2\pi {\rm i}j/m_T})\lambda(\alpha_T^*)\alpha_T$, 
where $\alpha_T^*\in \h_T$ is such that $\alpha_T(\alpha_T^*)=1$. 
This together with formula \eqref{sigmaf} implies the statement. 
\end{proof}

Let $c_B: \cS\to \bC $ be the function given by the formula 
\begin{equation*}
c_B(s)=-\frac{1}{2}\zeta_s a_B(s)\sum_{T\subset X^s}C(T,j(s)).
\end{equation*}
(the summation is over the connected components of $X^s$). 

\begin{corollary}{\label{prop:locell}}
Near $\lambda=0$, the elliptic Dunkl operators have the form:
$$
\cD_{v,C}^{\cL_\lambda}=-\sum_{s\in \cS}\frac{2c_B(s)\alpha_s(v)}{(1-\zeta_s)\alpha_s(B(\lambda))}s
+\text{\rm regular terms}.
$$
\end{corollary} 

\begin{remark}
Here we realize sections of line bundles on $X$ as functions on $\h$ with prescribed
periodicity properties under $\Gamma$. 
\end{remark}

\begin{remark}
Clearly, the same result applies to classical elliptic Dunkl operators. 
\end{remark}

\begin{proof}
The Corollary follows directly from Proposition \ref{asy} 
and the definition of $c_B(s)$. 
\end{proof}

\section{The main theorem}

\subsection{The statement of the main theorem}  

Define the operators 
$$
\overline{L}_i^{C,\lambda}:=P_i^{c_B}(\cD^{{\cL_\lambda}}_{\bullet,C},B(\lambda)),
$$
acting on local meromorphic sections of ${\cL_\lambda}$
(where $P_i^c(\bp,\bq)$ are the classical Calogero-Moser Hamiltonians, defined in Subsection \ref{ccmh}).
It is easy to see that these operators are independent on 
the choice of $B$ and commute with each other. 

Our main result is the following theorem. 

\begin{theorem}\label{maint}
(i) For any fixed $C$, the operators $\overline{L}_i^{C,\lambda}$
are regular in $\lambda$ near $\lambda=0$, 
and in particular have limits 
$\overline{L}_i^C$ as ${\cL}_\lambda$ tends 
to the trivial bundle (i.e., $\lambda$ tends to $0$).

(ii) The operators $\overline{L}_i^C$
are $G$-invariant and pairwise commuting 
elements of $\bC  G\ltimes D(X_{\reg})$. 

(iii) The restrictions $L_i^C$
of $\overline{L}_i^C$ to $G$-invariant meromorphic functions 
on $X$ are commuting differential operators 
on $X_{\reg}$, whose symbols are 
the polynomials $P_i$. 
\end{theorem}

\begin{definition} The commutative algebra generated by the 
collection of operators $\lbrace{L_i^C\rbrace}$ is 
called the quantum crystallographic
elliptic Calogero-Moser system attached to $G,X,C$.
\end{definition}

Note that only part (i) of Theorem \ref{maint} requires proof; 
once it is proved, parts (ii) and (iii) follow immediately. 
We will give two proofs of Theorem \ref{maint}(i). 
The first proof, given in Subsection \ref{pf1}, is based on Lemma \ref{lem:CCM}.
The second proof, given in Subsection \ref{pf2},  
is based on the techniques of \cite{BE} and reduction to
rank 1 (where the result can be proved by a direct calculation). 

\begin{remark} Theorem \ref{maint}(iii) can be generalized
as follows: for any character $\chi$ of $G$, 
the restrictions $L_i^{C,\chi}$
of $\overline{L}_i^C$ to $G$-equivariant meromorphic functions 
on $X$ which change according to $\chi$ under the $G$-action
are commuting differential operators 
on $X_{\reg}$, whose symbols are 
the polynomials $P_i$. Moreover, similarly to 
the results of \cite{BC} in the rational case, one can show 
that $L_i^{C,\chi}=L_i^{C+\delta_\chi}$, 
where $\delta_\chi$ is a certain shift of parameters.
\end{remark}

\subsection{The classical version of the main theorem}
The quantum system of Theorem \ref{maint} 
can be easily degenerated to a classical integrable system, 
by replacing elliptic Dunkl operators with their classical counterparts. 
Namely, define  
$$
\overline{L}_{i}^{0,C,\lambda}:=P_i^{c_B}(\cD^{{0,\mathcal
L_\lambda}}_{\bullet,C},B(\lambda)),
$$

\begin{theorem}\label{maint1}
(i) For any fixed $C$, the elements $\overline{L}_i^{0,C,\lambda}$
are regular in $\lambda$ near $\lambda=0$, 
and in particular have limits $\overline{L}_i^{0,C}$ as ${\cL}_\lambda$ tends 
to the trivial bundle (i.e., $\lambda$ tends to $0$).

(ii) The elements $\overline{L}_i^{0,C}$ are $G$-invariant and 
belong to $\bC  G\ltimes \cO(T^* X_{\reg})$. 

(iii) The functions $L_i^{0,C}:={\bm}(\overline{L}_i^{0,C})$ 
are Poisson commuting regular functions 
on $T^*X_{\reg}$, whose leading terms in momentum variables 
are the polynomials $P_i(\bp)$. 
\end{theorem}

Theorem \ref{maint1} is proved analogously 
to Theorem \ref{maint}, and can also be deduced from it 
by taking the quasiclassical limit. 

\begin{definition} The algebra generated by the 
collection of functions $\lbrace{L_i^{0,C}\rbrace}$ is 
called the classical crystallographic
elliptic Calogero-Moser system attached to $G,X,C$.
\end{definition}

\subsection{Examples and remarks}

\begin{example}\label{weylgp}
Let $\Gamma_\tau\subset \bC $ be a lattice generated by $1$ and 
$\tau\in \bC _+$. Let $E_\tau=\bC /\Gamma_\tau$ be the corresponding 
elliptic curve. Let $R$ be a reduced irreducible root system,
and $P^\vee$ be the coweight lattice
of $R$. Let $G=W$ be the Weyl group of $R$.
Let $X=E_\tau\otimes P^\vee$. In this case, the reflections $s_\alpha$ 
correspond to positive roots $\alpha\in R_+$, and we will write $T_\alpha$ for $T_{s_\alpha}$.  
It is easy to see that the elliptic curve $X/T_{\alpha}$
is naturally identified with $E_\tau$ via the map $\alpha: X/T_{\alpha}\to E_\tau$.

Let $(~,~)$ be the $W$-invariant inner product on $\h^*$,
normalized by the condition that the long roots
have squared length $2$. It is easy to see 
from the above that one can uniquely choose $B$
so that  
$$
a_B(s_{\alpha})=(\alpha,\alpha). 
$$

Assume first that $C(T,1)=0$ unless
$T$ passes through the origin (e.g., this happens automatically if $X^{s_\alpha}$ is connected for all roots $\alpha$). 
Let $C_\alpha=C(T_\alpha,1)$. Then we have
$c_\alpha:=c_B(s_\alpha)=C_\alpha(\alpha,\alpha)/2$ (so in the simply laced case, $c_\alpha=C_\alpha$).
In this case, $P_1(\bp)=(\bp,\bp)$, 
and the corresponding differential operator
$L_1^C$ is the elliptic Calogero-Moser operator
\begin{equation}
L_1^C=\Delta_\h-\sum_{\alpha>0}C_\alpha(C_\alpha+1)(\alpha,\alpha)\wp((\alpha,\bx),\tau),
\end{equation}
where $\Delta_\h$ is the Laplace operator defined by $(~,~)$,
and $\wp$ is the Weierstrass function. 

It remains to consider the case when $X^{s_\alpha}$ is disconnected for some $\alpha$, and
$C(T,1)$ can be nonzero for $T$ not necessarily passing through $0$. 
This happens only in type $B_n, n\ge 1$, for short roots $\alpha$. (Here $B_1=A_1$, 
but we use the normalization of the form given by $(\alpha,\alpha)=1$.) In this case, 
$X=E_\tau^n$, and $s_\alpha$ negates the $i$-th coordinate for some $i=1,\ldots,n$, so there are 
4 components of $X^{s_\alpha}$: $\alpha(\bx)=\xi_l$, $l=1,2,3,4$, where 
$\xi_1=0,\xi_2=1/2,\xi_3=\tau/2,\xi_4=(1+\tau)/2$ are the points of order $2$ on $E_\tau$.  
Let us denote the values of $C$ corresponding to these components by $C_l$. 
Then $c_\alpha=(C_1+C_2+C_3+C_4)/2$, and denoting by $k$ the value of $C$ 
for the long roots, we get 
\begin{eqnarray*}
L_1^C&=&
\sum_{i=1}^n \partial_i^2-\sum_{i\ne j}k(k+1)(\wp(x_i-x_j,\tau)+\wp(x_i+x_j,\tau))\\
&&\qquad-\sum_{l=1}^4\sum_{j=1}^n C_l(C_l+1)
\wp(x_j-\xi_l,\tau),
\end{eqnarray*}
which is the Hamiltonian of the 5-parameter Inozemtsev system \cite{I}.
For $n=1$ we have a 4-parameter generalization of the Lam\'e operator,
$$L= D-\sum_{l=1}^4 C_l(C_l+1)
\wp(z-\xi_l,\tau), D:=\frac{d}{dz},$$
which was first considered by Darboux (see \cite{Da}).
\end{example} 

\begin{remark}
The integrable systems $\lbrace L_i^C\rbrace$, 
$\lbrace L_i^{0,C}\rbrace$  
have rational limits, which are obtained as 
the lattice $\Gamma$ 
is rescaled by a parameter that goes to infinity. 
Namely, it is easy to see that these limits are the 
rational Calogero-Moser systems $\widehat{P}_i^{c'}$, 
$P_i^{c'}$, respectively, where 
\begin{equation*}
c'(s)=(1-\zeta_s)C(T_s,j(s))/2
\end{equation*}
(so for real reflection groups $c'(s)=C(T_s,j(s))$). 

However, the systems $\lbrace L_i^C\rbrace$, 
$\lbrace L_i^{0,C}\rbrace$ do not 
admit a trigonometric degeneration unless $G$ is a 
Weyl group. 
\end{remark}

\begin{remark}\label{elli}
If $G$ is a real reflection group, 
then instead of rational classical Calogero-Moser 
Hamiltonians, $P_i^c$, we could have used their {\it trigonometric
deformations}, and all the statements and proofs would 
carry over with obvious changes. Furthermore, 
for any $G$, instead of $P_i^c$ we could have used 
classical {\it elliptic} Calogero-Moser Hamiltonians associated to $G$ 
and the dual abelian variety $X^\vee$. In this case, the arguments 
of Section 5 show that if the parameters of these classical Hamiltonians 
are chosen appropriately, then the resulting operators 
$\overline{L}_i^{C,\lambda}$ are regular {\it for all} $\lambda$, 
not only for $\lambda=0$. This was conjectured by the authors 
of \cite{BFV} in 1994 (unpublished); for types $A_1$ and $A_2$,
this conjecture was confirmed by an explicit computation (see \cite{BFV}, 
pp. 908-909).\footnote{This construction of the crystallographic elliptic Calogero-Moser Hamiltonians 
is, in a sense, more natural than the one using the classical rational Hamiltonians, 
but unfortunately we could not have used it as the basic construction, since 
this would lead to a ``vicious circle'' (initially we don't have 
the elliptic Calogero-Moser Hamiltonians available even at the classical level).}  

\begin{example} Consider the type $BC_n$ case (the Inozemtsev system, 
Example \ref{weylgp}). It is easy to check that in this case
the appropriate choice of parameters for the ``dual'' classical system 
is as follows: 
\begin{eqnarray*}
& k'=k,\\
& C_1'=\frac{1}{2}(C_1+C_2+C_3+C_4)\\
& C_2'=\frac{1}{2}(C_1+C_2-C_3-C_4)\\
& C_3'=\frac{1}{2}(C_1-C_2+C_3-C_4)\\
& C_4'=\frac{1}{2}(C_1-C_2-C_3+C_4)
\end{eqnarray*}
(i.e., the function $C'$ is the Fourier transform of the function $C$ on the group 
of points of order $2$ on the elliptic curve). 
\end{example}
\end{remark}

\section{The main example} 

\subsection{The systems attached to groups $S_n\ltimes (\Bbb Z/m\Bbb Z)^n$}
Let $n$ be a positive integer, and $m=1,2,3,4$ or $6$. 
Then $G=S_n\ltimes (\bZ/m\bZ)^n$ is a complex 
crystallographic reflection group. Namely, 
$G$ acts on the torus $X=E_\tau^n$, where $E_\tau:=\bC /(\bZ\oplus \bZ\tau)$ 
is an elliptic curve, and $\tau$ is any point in $\bC _+$ 
for $m=1,2$, $\tau=e^{2\pi {\rm i}/3}$ for $m=3,6$, and $\tau={\rm i}$ 
for $m=4$. In this case, the above construction produces 
a quantum integrable system with Hamiltonians 
$L_1^C,\ldots,L_n^C$ ($G$-invariant differential operators on $E_\tau^n$ with meromorphic coefficients) 
such that 
$$
L_j^C=\sum_{i=1}^n \partial_i^{mj}+\text{l.o.t.},
$$
where l.o.t. stands for lower order terms. 
A similar construction involving classical 
counterparts of elliptic Dunkl operators 
yields a classical integrable system with 
Hamiltonians
$$
L_{j}^{0,C}=\sum_{i=1}^n p_i^{mj}+\text{l.o.t.}.
$$
In the case $m=1$, this system essentially reduces 
to the previous example (the Calogero-Moser system 
of type $A_{n-1}$). In the case $m=2$, 
it reduces to the 5-parameter Inozemtsev system, 
described above. 
However, for $m=3,4,6$, we get new integrable systems
with elliptic coefficients 
with cubic, quartic, and sextic lowest Hamiltonian, respectively.

The parameters of these systems are attached to the hypertori 
$x_i=x_j$ (a single parameter $k$) and to 
the hypertori $x_i=\xi$, where $\xi\in E_\tau$ 
is a point with a nontrivial stabilizer in $\bZ/m\bZ$
(the number of such parameters is the order 
of the stabilizer minus 1). For $m=3$, 
we have three fixed points $\xi$ of order $3$, 
for $m=4$ -- two fixed points of order $4$ and a fixed point
of order $2$, and for $m=6$ -- fixed points of orders
$2$, $3$, $6$, one of each (up to the action of $\bZ/m\bZ$).
Therefore, for $m=3$ this system has 
$7$ parameters, for $m=4$ it has $8$ parameters, and for 
$m=6$ it has $9$ parameters (if $n=1$,
the number of parameters drops by $1$, since the parameter $k$ is not present). 

Let us emphasize that the new crystallographic elliptic Calogero-Moser systems 
for $m>2$ exist only for special elliptic curves  
with additional $\bZ/m\bZ$-symmetry, which means that the corresponding $\wp$-function satisfies the equation
$$(\wp')^2 = 4\wp^3 -g_2 \wp -g_3$$
with either $g_3=0$ (the lemniscatic case, $\bZ/4\bZ$-symmetry) or $g_2=0$
(the equianharmonic case, $\bZ/3\bZ$-symmetry).

\subsection{The equianharmonic case with $m=3$} 

In the equianharmonic case with $m=3$, $\tau=e^{2\pi {\rm i}/3}$,  
we have the following proposition. 

\begin{proposition}\label{cobicop} The quantum Hamiltonian 
$L_1^C$ has the form 
\begin{eqnarray}
\label{cubiccase}
L_1^C&=& \sum_{i=1}^n \partial_i^3 +\sum_{i=1}^n
(a_0\wp(x_i)+a_1\wp(x_i-\eta_1)+a_2\wp(x_i-\eta_2))\partial_i \nonumber\\
&& -3k(k+1)\sum_{i<j}\sum_{p=0}^2 \wp(x_i-\varepsilon^p
x_j)(\partial_i+\varepsilon^{-p}\partial_j)\nonumber\\
&& +\sum_{i=1}^n(b_0\wp'(x_i)+b_1\wp'(x_i-\eta_1)+b_2\wp'(x_i-\eta_2)),
\end{eqnarray}
where $\tau=\varepsilon:=e^{2\pi {\rm i}/3}$, $\wp(x):=\wp(x,\tau)$, $\eta_1={\rm i}\sqrt{3}/3$, 
$\eta_2=-{\rm i}\sqrt{3}/3$, and $a_l,b_l,k$ are parameters.
\end{proposition} 

\begin{proof} $L_1^C$ must be a differential operator with meromorphic coefficients on
$E_\tau^n$ which satisfies the following conditions: 

1) The symbol of $L_1^C$ is $\sum \partial_i^3$;

2) $L_1^C$ is invariant under $S_n\ltimes \Bbb Z_3^n$;

3) the coefficients of order $3-r$ in $L_1^C$ 
are sums of meromorphic functions 
on $E_\tau^n$ with poles on the hypertori
$x_i=0,x_i=\eta_1,x_i=\eta_2$, $x_i=\varepsilon^p x_j$ 
($p=0,1,2$), and the sum of orders 
of all the poles being $\le r$.  

It is easy to see that the only operators with this property
are those of the form (\ref{cubiccase}).
\end{proof}

Thus, the one-dimensional operator 
corresponding to the case $m=3$ 
has the form
\begin{eqnarray}
\label{cubiccase1}
L&=& D^3 +\sum_{i=1}^n
(a_0\wp(z)+a_1\wp(z-\eta_1)+a_2\wp(z-\eta_2))D \nonumber\\
&& +\sum_{i=1}^n(b_0\wp'(z)+b_1\wp'(z-\eta_1)+b_2\wp'(z-\eta_2)),
\end{eqnarray}
where $D=\frac{d}{dz}.$

\subsection{The lemniscatic case} In the lemniscatic case, 
with $m=4$ and $\tau={\rm i}$, 
the corresponding one-dimensional operator has the following explicit form:
$$
L=D^4 +[a_0\wp(z) +a_1\wp(z-\omega_1)-2k(k+1)(\wp(z-\omega_2)+\wp(z-\omega_3))]D^2
$$
$$
+[b_0\wp'(z) +b_1\wp'(z-\omega_1)-2k(k+1)(\wp'(z-\omega_2)+\wp'(z-\omega_3))]D
$$
$$
+[k(k+1)(k+3)(k-2)(\wp^2(z-\omega_2)+\wp^2(z-\omega_3))
$$
$$
+k(k+1)(a_0-a_1)\wp(\omega_3)(\wp(z-\omega_2)-\wp(z-\omega_3))
$$
 \begin{equation}
 \label{lemn}+c_0\wp^2(z) +c_1\wp^2(z-\omega_1)],
 \end{equation}
where $\omega_1=(1+{\rm i})/2, \omega_2={\rm i}/2, \omega_3= 1/2$ and $k, a_0,a_1, b_0, b_1, c_0, c_1$ are arbitrary parameters. 

\section{Proofs of the Main Theorem}{\label{sec:proof}}

In this section, we give two different proofs of Theorem \ref{maint} (i).

\subsection{The first proof of Theorem \ref{maint}}\label{pf1}

For simplicity of exposition, we will work in a neighborhood 
$U$ of $0$ in $X$ (or, equivalently, in $\h$), which allows us to naturally trivialize 
the bundles $\cL_\lambda$, and regard sections of all line bundles 
as ordinary functions. 

For $v\in \h$, define an operator on the space of meromorphic functions 
of ${\bx}$ and $\lambda$ by the formula 
$$
(\cE_{v,C}F)({\bx},\lambda)=(\cD_{v,C}^{\cL_\lambda}+\sum_{s\in \cS}
\frac{2c_B(s)\alpha_s(v)}{(1-\zeta_s)\alpha_s(B(\lambda))}(s\otimes s^\vee))
F({\bx},\lambda),
$$
where $(s^\vee F)({\bx},\lambda):=F({\bx},s^{-1}\lambda)$. 

\begin{proposition}\label{pro1}
The operators $\cE_{v,C}$ commute, i.e. 
$[\cE_{v,C},\cE_{v',C}]=0$ for all $v,v'\in \h$.  
\end{proposition}

\begin{proof}
By Proposition \ref{prope}, 
the operators $\cD_{v,C}^{\bullet}, v\in \h$, linear functions $\psi(B(\lambda))$, $\psi\in \h^*$, and the operators $s\otimes s^\vee$ satisfy 
the defining relations of the algebra $\bC  G\ltimes S(\h\oplus \h^*)$. 
This implies the desired statement, since the operators $\cE_{v,C}$ are exactly the classical Dunkl operators $\cbD^0_{\bullet,c_B}$
on these generators. 
\end{proof}

Set $\widetilde{L}_i^C:=P_i(\cE_{\bullet,C})$ (these operators make sense 
and are pairwise commuting by Proposition \ref{pro1}). 

\begin{proposition}\label{pro2}
One has $\widetilde{L}_i^C=\overline{L}_i^{C,\lambda}$. 
\end{proposition}

\begin{proof}
By Lemma \ref{lem:CCM}, $P_i(\cbD_{\bullet,c}^0)=P_i^c(\bp,\bq)$. 
Substituting $\cD_{\bullet,C}^{\cL_\lambda}$ instead of $\bp$
(which we can do by Proposition \ref{prope}) and replacing $\bold q$ by 
$B(\lambda)$, we get the desired equality. 
\end{proof}

\begin{corollary}\label{cor1}
The operators $\widetilde{L}_i^C$ are linear over functions of $\lambda$. 
\end{corollary}

\begin{proof} Follows immediately from Proposition \ref{pro2}.
\end{proof}

\begin{proposition}\label{pro3}
The operators $\cE_{v,C}$ map the space of functions which are regular 
in $\lambda$ near $\lambda=0$ to itself. 
\end{proposition}

\begin{proof}
By Corollary \ref{prop:locell}, 
near $\lambda=0$, the operator
$\cE_{v,C}$ has the form
$$
\sum_{s\in \cS}\frac{2c_B(s)\alpha_s(v)}{(1-\zeta_s)\alpha_s(B(\lambda))}s\otimes (s^\vee-1)+\text{\rm regular terms }.
$$
Since the operator $\frac{1}{\alpha_s(B(\lambda))}(s^\vee-1)$ preserves regularity in $\lambda$, 
the statement follows. 
\end{proof}

By Proposition \ref{pro3}, 
the operators $\widetilde{L}_i^C$ preserve the space of functions 
which are regular in $\lambda$ near $\lambda=0$. 
By Corollary \ref{cor1}, this means that $\widetilde{L}_i^C$
are themselves regular in $\lambda$ for $\lambda$ near $0$. 
Hence, by Proposition \ref{pro2}, the 
operators $\overline{L}_i^{C,\lambda}$ are regular in $\lambda$ near $\lambda=0$, 
as desired. 

\subsection{Relation to Cherednik's proof}

In this subsection we would like to explain the connection between
the construction of Subsection \ref{pf1} and Cherednik's proof of 
the integrability of elliptic Calogero-Moser systems attached to Weyl groups
(\cite{Ch2}). 

Recall that to obtain the operators $\cE_{v,C}$ used in Subsection \ref{pf1} 
from the elliptic Dunkl operators $\cD_{v,C}^{\cL_\lambda}$, we
``subtract'' the pole in $\lambda$ by adding the 
reflection part of the rational Dunkl operator 
with respect to $\lambda$. 

As we mentioned in Remark \ref{elli}, 
in the real reflection group case, instead of the rational Dunkl operator
we could have used the trigonometric one. 
Let us denote the corresponding operators by $\cE_{v,C}^{\rm trig}$. 

For $\lambda\in \Hom(P^\vee,\bC ^*)=\h^\vee/Q$, denote by ${\mathcal F}_\lambda$ the space of meromorphic functions on $\h$ 
which are periodic under $P^\vee$ and transform by a character under $\tau P^\vee$, representing  
sections of $\cL_\lambda$. Let ${\mathcal F}=\oplus_{\lambda\text{ regular }}{\mathcal F}_\lambda$. 
It is easy to check that the operator $\cE_{v,C}^{\rm trig}$ acts naturally on ${\mathcal F}$. 

On the other hand, in \cite{Ch2}, Cherednik defined affine Dunkl operators, $\cD_{v,C}^{\rm aff}$
(\cite{Ch2}, formula (3.4) after specialization of the central element). These are differential-difference operators on functions on $\h/P^\vee$
(involving shifts by elements of $\tau P^\vee$ composed with reflections), which preserve 
the space ${\mathcal F}$. 

It turns out that the operators $\cE_{v,C}^{\rm trig}$ and $\cD_{v,C}^{\rm aff}$
on the space ${\mathcal F}$ coincide. This shows that in the real reflection group case,
the construction of Subsection \ref{pf1} is, essentially, a modification of 
the construction of \cite{Ch2}. 

\subsection{The second proof of Theorem \ref{maint}}\label{pf2}

\begin{proposition}\label{pro4}
Theorem \ref{maint}(i) holds in rank 1, i.e., if 
$\dim X=1$. 
\end{proposition}

\begin{proof}
In the rank 1 case, $G=\bZ/m\bZ$. 
Let $\mathbb C$ be the reflection representation of $G$
with coordinate function $x$.
Let $g$ be the generator of $G$ acting on $\Bbb C$ by 
multiplication by $\xi=e^{2\pi {\rm i}/m}$ (i.e., $gx=\xi^{-1}x$).
The primitive idempotents of $\Bbb CG$ are defined by
$$
e_{i}=\frac{1}{m}\sum_{j=0}^{m-1}\xi^{ij}g^{j},\ i=0,...,m-1;
$$
they satisfy the relations
$$
e_{i}e_{j}=\delta_{ij}e_{i}, \quad \sum_{i=0}^{m-1}e_{i}=1.
$$

We also have the following cross relations (the indexing is modulo $m$):
$$
e_{i}x=xe_{i-1}, e_{i}\partial_{x}=\partial_{x}e_{i+1},
e_{i}p=pe_{i+1} (p \text{ is the symbol of }\partial_{x}).
$$

For brevity, we will abuse notation and write $\lambda$ instead of $B(\lambda)$. 
From Corollary \ref{prop:locell}, we know that near $\lambda=0$
the elliptic Dunkl operator can be written as
$$
\cD_{v,C}^{\cL_\lambda}=\partial_{x}
+\frac{1}{\lambda}\sum_{i=0}^{m-1} b_{i}e_{i}+
\sum_{j=0}^{m-1} R_{j}\sum_{i=0}^{m-1} b_{i}e_{i},
$$
where $\sum b_{i} = 0$, $\mathbf{b} = (b_{0},\ldots,b_{m-1})$ 
is related to $c_{B}$ by a certain invertible linear transformation, and $R_{j}$ has the form
$$
R_{j}=\mathop{\mathop{\sum_{s\ge -1,t \geq 0}}
_{s\equiv j \,\mathrm{mod}\, m}}_{s+t\equiv-1 \,\mathrm{mod}\, m} 
a_{st}x^{s}\lambda^{t},\quad \text{ where }a_{st} \text{ are constants}.
$$
So we have $R_{j}e_{i}=e_{i+j}R_{j}$. Here all indices are modulo $m$.

We have $P_1=P=p^m$, and 
\begin{equation}\label{rank1}
P^{c_B}(p,q)=(p-\frac{b_0}{q})\cdots(p-\frac{b_{m-1}}{q}).
\end{equation} 

Define $\Phi_{i}(p,q)=p-\frac{b_i}{q}$
and $\Phi_{i}=\Phi_{i}(\cD_{v,C}^{\cL_\lambda}, \lambda)$.

\begin{lemma}{\label{lem:reg}}
For any integer $r,s$ with $1\leq s\leq m$, the expression
$$
\Phi_{r+1}\cdots \Phi_{r+s}e_{r+s}
$$ 
is regular in $\lambda$ at $\lambda=0$.
\end{lemma}
\begin{proof}
We prove the statement by induction on $s$.
By a direct computation, one can see 
that the statement is true when $s=1$.
Now for the induction step suppose the statement is true for $s<k$, where $k\ge 2$, 
and let us prove it holds for $s=k$.
We have
$$\Phi_{r+1}\cdots \Phi_{r+k}e_{r+k}
$$
$$
=\Phi_{r+1}\cdots \Phi_{r+k-1}e_{r+k-1}(\partial_{x}+b_{r+k}R_{m-1})
+b_{r+k}\sum_{j=2}^{m}\Phi_{r+1}\cdots \Phi_{r+k-1} e_{r+k-j}R_{m-j}.
$$

Notice that $R_{m-j}=\lambda^{j-1}R'_{m-j}$ for $j=1,...,m$, where $R'_{j}$  
is regular at $\lambda=0$, and $\Phi_{i}$ has only a simple pole in $\lambda$.
So $j=2,...,k-1$ we have
\begin{eqnarray*}
&&\Phi_{r+1}\cdots \Phi_{r+k-1} e_{r+k-j}R_{m-j}\\
&=&\Phi_{r+1}\cdots \Phi_{r+k-j}
(\Phi_{r+k-j+1}\lambda)
\cdots
(\Phi_{r+k-1}\lambda) e_{r+k-j}R'_{m-j}\\
&=&(\Phi_{r+k-j+1}\lambda)
\cdots
(\Phi_{r+k-1}\lambda)\Phi_{r+1}\cdots \Phi_{r+k-j}
 e_{r+k-j}R'_{m-j},
\end{eqnarray*}
which is regular in $\lambda$ by the induction hypothesis, 
while for $k\le j\le m$ the above expression is regular 
since $R_{m-j}$ is divisible by $\lambda^{j-1}$. 
Also, the expression 
$$
\Phi_{r+1}\cdots \Phi_{r+k-1}e_{r+k-1}(\partial_{x}+b_{r+k}R_{m-1})
$$
is regular by the induction hypothesis. The induction step is thus completed. 
\end{proof}

Now since
$$
\overline{L}_i^{C,\lambda}
=\prod_{i=0}^{m-1}\Phi_{i}=
\sum_{j=0}^{m-1}\Phi_{j-m+1}\cdots\Phi_{j}e_{j},
$$
Lemma \ref{lem:reg} implies that the operators $\overline{L}_i^{C,\lambda}$ are regular in $\lambda$ near $\lambda = 0$.
\end{proof}


Now we proceed to prove Theorem \ref{maint}(i) in rank $n>1$. 
By Hartogs' theorem, it suffices to check the regularity of 
$\overline{L}_i^C$ at a generic point of 
a reflection hyperplane $H\subset \h$. 
To this end, we will use the following 
proposition. 

Let $H$ be a reflection hyperplane in $\h$. 
Let $s\in \cS$ be a generator of $G_{H}\cong \bZ_m$.
Let $p\in \h_s$, $q\in \h_s^*$ be 
such that $(p,q)=1$, and let $p_1,\ldots,p_{n-1}$, 
$q_1,\ldots,q_{n-1}$ be bases of $\h^s$, $(\h^s)^*$.  
Also, since the 1-dimensional space $\h_s$ 
carries a $G_H$-action, we can define
the classical Calogero-Moser Hamiltonian 
$P^c(p,q)$ (given by formula \eqref{rank1}). 

Let $\bx_0$ be a generic point of $H$, 
and let $q_i^0:=q_i(\bx_0)$. (Note that $q(\bx_0)=0$.)

\begin{proposition}{\label{lem:loc}}
Near a generic point $\bx_0$ of $H$, for any $i=1,\ldots,n$, 
the function $P_i^c$ can be written as a polynomial of 
the functions $p_1,\ldots,p_{n-1}$, $pq$, and  
$P^c(p,q)$, whose coefficients are 
power series in the functions 
$q_1-q_1^0$,\ldots,$q_{n-1}-q_{n-1}^0$, $q^m$. 
\end{proposition}

\begin{proof} Let 
\begin{equation}\label{symed}
\e=\frac{1}{|G|}\sum_{g\in G}g\in \bC  G
\end{equation}
be the symmetrizing idempotent. 
The function $P_i^c$ belongs to the spherical subalgebra $B_{0,c}(G,\h):=\e H_{0,c}(G,\h)\e$
of the rational Cherednik algebra $H_{0,c}(G,\h)$ (sitting inside $\cO(T^*\h_{\reg})^G$). 
By the classical version of Theorem 3.2 of \cite{BE} (see also \cite{B}),
the completion at $\bx_0$ of the algebra $B_{0,c}(G,\h)$ is isomorphic to 
the completion at $0$ of the algebra $\bC [q_1,\ldots,q_{n-1},p_1,\ldots,p_{n-1}]\otimes B_{0,c}(G_H,\h_s)$. 
However, the algebra $B_{0,c}(G_H,\h_s)$ is generated by $q^m,pq$ and $P^c(p,q)$. 
This implies the desired statement.  
\end{proof}

Theorem \ref{maint}(i) follows immediately from Proposition \ref{pro4} and Proposition \ref{lem:loc}.

\section{A geometric construction of quantum crystallographic elliptic Calogero-Moser systems}

In this section we will give a geometric construction of the
quantum crystallographic elliptic Calogero-Moser systems described above, in the style
of the Beilinson-Drinfeld construction of the quantum Hitchin
system, \cite{BD}. Namely, we construct these systems as 
algebras of global sections of sheaves of spherical elliptic
Cherednik algebras, for the critical value of the twisting
parameter. On the other hand, if the twisting parameter is not critical, we show
that the algebra of global sections reduces to 
$\bC $. 

\subsection{Cherednik algebras of varieties with a finite group
action}

Let us recall the basics on the Cherednik algebras of varieties
with a finite group action, introduced in \cite{E2} (see also
\cite{EM2}, Section 7). 

Let $X$ be a smooth affine algebraic variety over 
$\bC$. For a closed hypersurface $Y\subset X$,
let $\cO_X(Y)$ be the space of regular functions on
$X\setminus Y$ with a pole of at most first order on $Y$.
Let $\xi_Y: {\rm Vect}(X)\to \cO_X(Y)/\cO_X$ be the
natural map. 

Let $G$ be a finite group of automorphisms of $X$. 
Let  $\cY$ be the set of pairs $(Y,s)$,
where $s\in G$, and $Y$ is a connected component of 
the set of fixed points $X^{s}$ such that $\codim Y=1$
(called a {\it reflection hypersurface}).
Let $\lambda_{Y,s}$ be the eigenvalue of $s$ on the 
conormal bundle of $Y$. Let $X_{\reg}$ be the complement of reflection 
hypersurfaces in $X$. 

Fix $\omega\in \rH^2(X)^G$, and let $D_\omega(X)$ 
be the algebra of twisted differential operators  
on $X$ with twisting $\omega$. 

Let $c: \cY\to \bC$ be a $G$-invariant function.
Let $v$ be a vector field on $X$, 
and let $f_Y\in \cO_X(Y)$ 
be an element of the coset $\xi_Y(v)\in {\mathcal
O}_X(Y)/\cO_X$. 

A {\it Dunkl-Opdam operator} for $G,X$ is an operator 
given by the formula 
\begin{equation*}
\cD:={\bL}_v+\sum_{(Y,s)\in
\cY}f_Y\cdot \frac{2c_{Y,s}}{1-\lambda_{Y,s}}
(s-1),
\end{equation*}
where ${\bL}_v\in D_\omega(X)$ is the $\omega$-twisted Lie
derivative along $v$ (here we pick a closed $2$-form representing
$\omega$).

The {\it Cherednik algebra of $G,X$}, 
$H_{1,c,\omega}(G,X)$, is  generated  
inside $\Bbb CG\ltimes D_\omega(X_{\reg})$ by the function algebra 
$\cO_X$, the group $G$, and the Dunkl-Opdam 
operators $\cD$. 

Now let $X$ be any smooth algebraic variety (not necessarily affine), 
and let $G$ be a finite group acting on $X$.
Assume that $X$ has a $G$-invariant affine open covering,
so that $X/G$ is also a variety.
Recall that twistings of differential operators on $X$
are parametrized by $\rH^{2}(X, \Omega_{X}^{\geq 1})$; 
in particular, if $X$ is projective, they are parametrized by 
$\rH^{2,0}(X)\oplus \rH^{1,1}(X)$ (see \cite{BB}).
So for $\psi\in \rH^{2}(X, \Omega_{X}^{\geq
1})^{G}$, we can define the sheaf of Cherednik algebras
$H_{1, c,\psi, G,X}$ (a quasicoherent sheaf on $X/G$), 
by gluing the above constructions on $G$-invariant affine open 
sets. Namely, for an affine open set $U\subset X/G$, we set
$$
H_{1, c,\psi, G,X}(U):=H_{1, c,\psi}(G,\overline{U}), 
$$ 
where $\overline{U}$ is the preimage of $U$ in $X$.
We can also define the sheaf of spherical Cherednik algebras,
$B_{1, c,\psi, G,X}$, given by 
$$
B_{1, c,\psi, G,X}(U)=\e H_{1, c,\psi}(G,\overline{U})\e
$$
where $\e$ is the symmetrizing idempotent of $G$, 
defined by \eqref{symed}. 

Finally, let us define the sheaves of modified Cherednik algebras, 
$H_{1,c,\psi,\eta,G,X}$ and modified spherical Cherednik algebras
$B_{1,c,\psi,\eta,G,X}$. Let 
$\eta$ be a $G$-invariant function on the set of reflection
hypersurfaces in $X$. Define a modified Dunkl-Opdam operator
for $G,X$ (when $X$ is affine) by the formula
\begin{equation*}
\cD:={\bL}_v+\sum_{(Y,s)\in
\cY}\frac{2c_{Y,s}}{1-\lambda_{Y,s}}
f_Y\cdot (s-1)+\sum_{Y}\eta(Y)f_Y,
\end{equation*}
(where the summation in the second sum is over all reflection 
hypersurfaces), and define the sheaf of algebras 
$H_{1,c,\psi,\eta,G,X}$ to be locally generated 
by $\cO_X$, $G$, and modified Dunkl-Opdam operators
(so, we have $H_{1,c,\psi,0,G,X}=H_{1,c,\psi,G,X}$). 
Also, set $B_{1,c,\psi,\eta,G,X}:=\e H_{1,c,\psi,\eta,G,X}\e$. 

Note that according to the PBW theorem, 
the sheaf $H_{1,c,\psi,\eta,G,X}$ has an 
increasing filtration $F^\bullet$, such that 
${\gr}(H_{1,c,\psi,\eta,G,X})=G\ltimes \cO_{T^*X}$. 

Note also that the modified Cherednik algebras can be expressed via the
usual ones (see \cite{E2}, \cite{EM2}). 
Namely, let $\psi_Y$ be the twisting of differential operators 
on $X$ by the line bundle $\cO_X(Y)^*$. 
Then one has 
$$
H_{1,c,\psi,\eta,G,X}\cong H_{1,c,\psi+\sum_Y \eta(Y)\psi_Y,G,X}.
$$

Finally, note that we have a canonical isomorphism of sheaves 
$$
H_{1, c, 0,\eta,G,X}|_{X_{\reg}}\cong 
\bC  G\ltimes D_{X_{\reg}}.
$$ 

\subsection{Elliptic Cherednik algebras and crystallographic elliptic Calogero-Moser systems}

Now let $X$ be an abelian variety, and $G$ an irreducible 
complex reflection group acting on $X$, as in Section 2. 
It is easy to see that $(\wedge^2\h^*)^G=0$, so $X$ does not 
admit nonzero global 2-forms. This implies that 
the space of $G$-invariant twistings of differential operators 
on $X$ is $\rH^{1,1}(X)^G$, which is 1-dimensional, 
and spanned by the K\"ahler form on $X$ defined by the Hermitian form
$B$. So we can make the identification $\rH^{1,1}(X)^G\cong \bC $. 

It is well known that $X$ admits a $G$-invariant affine open
covering, so $X/G$ is an algebraic variety, and we can consider the sheaves 
$H_{1, c,\psi,\eta, G,X}$ and $B_{1, c,\psi,\eta, G,X}$ on $X/G$.

Notice that we have an isomorphism $\cY\cong \cA$.
Thus we can substitute for $c$ the function 
$$
c_{T,s}=(1-e^{-2\pi {\rm i}j(s)/m_s})C(T,j(s))/2.
$$ 
Also, define a function $\eta_C$ on the set of reflection hypertori 
by the formula 
$$
\eta_C(T):=\sum_{j=1}^{m_T-1}C(T,j).
$$

The main result of this section is the following theorem, 
which gives a geometric construction of the quantum elliptic
integrable systems.  

\begin{theorem}\label{glosec}
(i) Restriction to $X_{\reg}$ defines 
an isomorphism 
$$
\Gamma(X/G,B_{1,c,0,\eta_C,G,X})\cong 
\bC [L_1^C,\ldots,L_n^C].
$$

(ii) The algebra of global sections 
$\Gamma(X/G,B_{1,c,\psi,G,X})$ 
is nontrivial (i.e. not isomorphic to $\bC $)
if and only if
\begin{equation}\label{crit}
\psi=\sum_{(T,j)\in \cA} C(T,j)\psi_T. 
\end{equation}
If \eqref{crit} holds, 
$\Gamma(X/G,B_{1,c,\psi,G,X})$ 
is a polynomial algebra in generators $L_i$ whose symbols are $P_i$.
\end{theorem}

\begin{example}
If $C=0$, Theorem \ref{glosec} states that for $\psi\in \bC $, 
there exist nontrivial $G$-invariant $\psi$-twisted global differential
operators on $X$ if and only if $\psi=0$, in which case the algebra of 
such operators is $(S\h)^G$. This is, of course, easy to check directly. 
\end{example}

\subsection{Proof of Theorem \ref{glosec}}

We first prove (i). The sheaf of algebras 
$H_{1,c,0,\eta_C,G,X}$ is locally generated by regular functions on $X$,
elements of $G$, and Dunkl-Opdam operators without a ``pure
function'' term: 
\begin{equation*}
\cD={\bL}_v+\sum_{(Y,s)\in
\cY}f_Y\cdot \frac{2c_{Y,s}}{1-\lambda_{Y,s}}
s.
\end{equation*}
This implies that for a generic $\cL$, 
the elliptic Dunkl operators  
$\cD_{v,C}^{\cL}$ are sections 
of the sheaf $H_{1,c,0,\eta_C,G,X}$ 
on the formal neighborhood of any point in $X/G$. 
Thus, the same applies to the operators 
$\overline{L}_i^{C,\lambda}$, and hence to their limits at
$\lambda=0$, $\overline{L}_i^C$ (which exist by Theorem
\ref{maint}). But since the coefficients 
of $\overline{L}_i^C$ are periodic, 
$\overline{L}_i^C$ are actually {\sl global} sections of 
the sheaf $H_{1,c,0,\eta_C,G,X}$. Thus, 
$L_i^C$ are global sections of $B_{1,c,0,\eta_C,G,X}$, i.e., 
$\bC [L_1^C,\ldots,L_n^C]\subset
\Gamma(X/G,B_{1,c,0,\eta_C,G,X})$. To see that this inclusion is 
an isomorphism, it suffices to show that it is an isomorphism 
for the corresponding graded algebras, which is obvious, 
since $\Gamma(X/G,{\rm gr}(B_{1,c,0,\eta_C,G,X}))=(S\h)^G$. 

Now we prove (ii). As explained above, we have an isomorphism 
$$
H_{1,c,0,\eta_C,G,X}\cong H_{1,c,\sum_{(T,j)\in \cA}
C(T,j)\psi_T,G,X},
$$
which proves the ``if'' part of (ii). It remains to prove the
``only if'' part, i.e. that if equation \eqref{crit} does not hold
then the algebra of global sections is trivial. 
To this end, for $r\ge 1$ consider the vector bundle 
$$
\cE:=F^r H\e/F^{r-2}H\e,
$$
on $X$, where $H=H_{1,c,\psi,G,X}$. 
We have an exact sequence of vector bundles on $X$:
$$
0\to S^{r-1}\h\to \cE\to S^r\h\to 0,
$$
where the bundles $S^k\h$ are trivial. 
Such an extension is determined by 
an extension class $\beta$ in 
\begin{eqnarray*}
\Ext^1(S^r\h,S^{r-1}\h)&=&
\Hom_{\bC }(S^r\h,S^{r-1}\h)\otimes 
\Ext^1(\cO_X,\cO_X)\\
&=&
\Hom_{\bC }(S^r\h,S^{r-1}\h\otimes \h). 
\end{eqnarray*}
(since $\Ext^1(\cO_X,\cO_X)=\h$). 
A direct calculation shows that (up to a nonzero constant) 
$\beta$ is the canonical inclusion multiplied by the number
$$
\psi-\sum_{(T,j)\in \cA} C(T,j)\psi_T. 
$$
So if \eqref{crit} does not hold, $\beta$ is injective, and thus 
no nonzero section of $S^r\h$ can be lifted to a section of
$\cE$. This implies the ``only if'' part of (i). 

\subsection{The classical analog of Theorem \ref{glosec}}

In this subsection we give a geometric construction of the classical crystallographic elliptic Calogero-Moser systems. 

Define a modified classical Dunkl-Opdam operator
for $G,X$ (when $X$ is affine) by the formula
\begin{equation*}
\cD^0:=p_v+\sum_{(Y,s)\in
\cY}\frac{2c_{Y,s}}{1-\lambda_{Y,s}}
f_Y\cdot (s-1)+\sum_Y \eta(Y)f_Y.
\end{equation*}
Let $T_\psi^*X$ denote the $\psi$-twisted cotangent 
bundle of $X$ (see \cite{BB}, Section 2), and define the sheaf of modified classical elliptic Cherednik 
algebras $H_{0,c,\psi,\eta,G,X}$ to be locally generated 
inside $\bC  G\ltimes \cO(T^*_\psi X_{\reg})$ 
by $\cO_X$, $G$, and modified classical Dunkl-Opdam operators (\cite{E2}).
The ``unmodified'' version $H_{0,c,\psi,0,G,X}$ will be shortly denoted by 
$H_{0,c,\psi,G,X}$. Also, set $B_{0,c,\psi,\eta,G,X}:=\e H_{0,c,\psi,\eta,G,X}\e$. 

\begin{theorem}\label{glosec1}
(i) Restriction to $X_{\reg}$ defines 
an isomorphism 
$$
\Gamma(X/G,B_{0,c,0,\eta_C,G,X})\cong 
\bC [L_1^{0,C},\ldots,L_n^{0,C}].
$$

(ii) The algebra of global sections 
$\Gamma(X/G,B_{0,c,\psi,G,X})$ 
is nontrivial (i.e. not isomorphic to $\bC $)
if and only if
\begin{equation}\label{crit1}
\psi=\sum_{(T,j)\in \cA} C(T,j)\psi_T. 
\end{equation}
If \eqref{crit1} holds, 
$\Gamma(X/G,B_{0,c,\psi,G,X})$ 
is a polynomial algebra in generators $L_i^{(0)}$ 
whose leading terms in momentum variables are $P_i$.
\end{theorem}

\begin{proof}
The proof is parallel to the proof of Theorem \ref{glosec}, 
using Theorem \ref{maint1}.
\end{proof}

\section{Algebraic integrability of quantum crystallographic elliptic Calogero-Moser systems}

Let $\lbrace{L_1,\ldots,L_n\rbrace}$ be a quantum integrable system (i.e., a commuting system of differential operators) 
on an open set $U\subset \bC ^n$. Assume that the symbols $P_i$ of $L_i$ have constant coefficients, 
and $\bC [p_1,\ldots,p_n]$ is a finitely generated module 
(of some rank $r$) over $\bC [P_1,\ldots,P_n]$.  
Consider the joint eigenvalue problem:
\begin{equation}\label{algint}
L_i\Psi=\Lambda_i\Psi. 
\end{equation}
Clearly, the space of local holomorphic solutions of this system 
near a generic point $\bx_0\in U$ has dimension $r$. Recall 
\cite{K2, CV1,CV2, BEG} that the system $\lbrace{L_i\rbrace}$ 
is said to be {\it algebraically integrable} if there exists a differential operator $L$ 
on $U$ which commutes with $L_i$ and acts with distinct eigenvalues on 
the space of local solutions of \eqref{algint} for generic $\Lambda_i$. 
In this case, the system of differential equations 
$$
L_i\Psi=\Lambda_i\Psi, L\Psi=\Lambda \Psi
$$ 
(where $\Lambda$ is a certain algebraic function of the $\Lambda_i$) 
can be reduced to a first order scalar system, and thus the 
solutions of system \eqref{algint} can (in principle) be written explicitly
in quadratures.

It was proved in \cite{CV1,VSC} that the rational and trigonometric quantum 
Calogero-Moser systems are algebraically integrable for any Weyl group if the 
parameters $c_\alpha$ are integers. The same result in the elliptic case was conjectured 
in \cite{CV1}\footnote{It is interesting that this conjecture was inspired by a remarkable result
of J. Ritt, who classified in dimension one all commuting rational maps in
terms of the symmetry groups of elliptic curves (see
\cite{Ve} and references therein)} 
and proved in \cite{CEO} (for type $A$, it was proved earlier in \cite{BEG}). 
It was also proved in \cite{CEO} that algebraic integrability 
holds for the Inozemtsev system with integer parameters. Finally, algebraic integrability of 
the rational quantum Calogero-Moser systems of complex reflection groups
was established recently in \cite{BC}.  

The following theorem establishes algebraic integrability 
of the crystallographic elliptic Calogero-Moser system
attached to any complex crystallographic reflection 
group, under an integrality assumption on the parameters. Namely, for any 
reflection hypertorus $T\subset X$ and any $l=0,1,\ldots,m_T-1$ 
define the number 
$$
m_l(T)=l+\sum_{j=1}^{m_T-1}C(T,j)e^{2\pi {\rm i}jl/m_T}.
$$

\begin{theorem}\label{ali}
If for all $l$ and $T$ the numbers $m_l(T)$ are integers which are pairwise distinct  
modulo $m_T$, then the quantum integrable system $\lbrace{L_i^C\rbrace}$ 
is algebraically integrable.
\end{theorem}

\begin{proof}
The proof is similar to the proof in the real reflection case, given in \cite{CEO}.

Namely, we first show that the holonomic system of differential equations 
\begin{equation}\label{sys}
L_i^C\Psi=\Lambda_i\Psi
\end{equation}
has regular singularities. This follows from the fact that 
\eqref{sys} is a limit as $\lambda\to 0$ 
of the eigenvalue problem 
for elliptic Dunkl operators 
\begin{equation}\label{sys1}
\cD_{v,C}^{\cL_\lambda,\nabla}\Psi=\Lambda(v)\Psi, 
\end{equation}
(\cite{EM1}) which obviously has 
regular singularities.
\footnote{Here it is important that we don't have moving poles. 
Otherwise (if poles are allowed to move and collide), 
a system with regular singularities can be degenerated to a system 
with irregular ones.} 

Thus, by Remark 3.10 of \cite{CEO}, it suffices to show that 
the monodromy of \eqref{sys} around the reflection hypertori is trivial. 
For the system \eqref{sys1},  
this property follows from the fact that this monodromy representation
factors through the orbifold Hecke algebra (see \cite{EM1}, Section 6.2);
indeed, since the parameters are integral and distinct modulo $m_T$, 
the orbifold Hecke algebra reduces to the group algebra of the orbifold fundamental group, 
implying the triviality of the monodromy.   
Now the required statement for \eqref{sys} follows by taking the limit 
$\lambda\to 0$.  
\end{proof}

\begin{corollary}\label{cube} The quantum integrable system defined by the operator 
(\ref{cubiccase}) is algebraically integrable if $k\in \Bbb Z$, and 
there exist integers $m_{ij}$, $i,j\in \lbrace{0,1,2\rbrace}$, pairwise non-congruent modulo $3$, 
with $m_{i0}+m_{i1}+m_{i2}=3$, such that 
$$
a_i=m_{i0}m_{i1}+m_{i0}m_{i2}+m_{i1}m_{i2}-2,
$$
and 
$$
b_i=\frac{1}{2}\prod_{j=0}^2 m_{ij}.
$$
for $i=0,1,2$. 
\end{corollary}

\begin{proof} The last condition means that the indices of the corresponding 1-variable 
operator $\partial^3+\frac{a_i}{x^2}\partial-\frac{2b_i}{x^3}$ 
(which we obtain by looking at the neighborhood of a generic point of the hypertori 
$x_j=0$, $x_j=\eta_1$, $x_j=\eta_2$ in $E_\tau^n$) are $m_{i0},m_{i1},m_{i2}$,
and thus are all integers. Now the absence of logarithmic terms (and hence, 
the triviality of monodromy) follows from the symmetry $x \rightarrow \varepsilon x, \varepsilon^3 =1$ 
and the fact that the indices $m_{i0},m_{i1},m_{i2}$ have different residues modulo $3$. 
\end{proof}

\begin{remark} The integrality and non-congruence assumptions in Theorem \ref{ali} (and in particular in 
Corollary \ref{cube}) are necessary. Here is a sketch of a proof. Suppose the system is algebraically integrable.
Let us translate the origin in $X$ to a generic point $\bold x$ of some reflection hypertorus 
$T\subset X$ with $m_T=m$, and then go to the rational limit by multiplying the lattice $\Gamma$ by a factor $K$ going to infinity.
Then we will get that the single-variable rational Calogero-Moser operator $L$ of order $m$ with the appropriate parameters is
algebraically integrable (this is seen by looking at what happens to the Dunkl-Opdam operators in the limit).
But in the single-variable rational case, it is known (and easy to prove) that the integrality and non-congruence conditions 
are necessary (see e.g. \cite{BC}). Namely, in this case the operator $L$ is homogeneous of degree $-m$, 
and in the algebraically integrable case it should have eigenfunctions (with eigenvalue $\mu^m$) of the form
$F(\mu x)$, where $F(x)=e^xQ(1/x)$, $Q$ being a polynomial, and it is easy to compute when there are such
solutions using the power series method. Since this argument can be applied to all reflection hypertori, 
it gives the integrality and non-congruence conditions for all the parameters. 
\end{remark}

\begin{example} 
Consider the case $n=1$, and $a_0=a_1=a_2=a$, $b_0=b_1=b_2=b$. 
Then Corollary \ref{cube} implies that the operator
$$
L=D^3+a\wp(z)D+b\wp'(z).
$$
(where $\wp(z)=\wp(z,\tau)$, $\tau=e^{2\pi {\rm i}/3}$) 
is algebraically integrable if 
there exists an triple of integers ${\bold m}=\lbrace m_0,m_1,m_2\rbrace$ ($m_0<m_1<m_2$), pairwise non-congruent modulo $3$, 
with $m_{0}+m_{1}+m_{2}=3$, such that 
$$
a=m_{0}m_{1}+m_{0}m_{2}+m_{1}m_{2}-2,
$$
and 
$$
b=\frac{1}{2}m_{0}m_1m_2.
$$

\begin{remark} In the case $m_2-m_1=m_1-m_0=n\in \Bbb N$,
the operator $L$ has the form
$$
L=D^3+(1-n^2)\wp(z)D+\frac{1-n^2}{2}\wp'(z).
$$
A proof of algebraic integrability of this operator
(i.e., of the existence of meromorphic eigenfunctions) 
in the equianharmonic case 
is given by Halphen, \cite{Ha}, p.571; this proof easily 
extends to general values of $m_i$.
\end{remark}
 
\begin{remark}
Note that if ${\bold m}=\lbrace 0,1,2\rbrace$, 
then $L=D^3$, so the algebraic integrability of $L$ is 
obvious, and if ${\bold m}=\lbrace -1,1,3\rbrace$, the algebraic 
integrability of $L$  follows from the fact that $L$ commutes 
with the Lam\'e operator $D^2-2\wp(z)$. The case ${\bold m}=(-1,0,4)$ 
is a special case of the algebraically integrable operator
$$
L=D^3-(6\wp(z)+c)D, c\in \Bbb C
$$
considered by Picard in 1881 (\cite{Pi}; see also \cite{For}, page 464, Ex. 13).

Observe that these examples are algebraically integrable for any elliptic curve. 
On the other hand, as explained in \cite{U}, 
if $\bold m=(-3,1,5)$, then the operator $L$ is algebraically 
integrable only in the equianharmonic case. 
\footnote{The general theory of algebraic integrability of operators 
of the form $D^3+(a\wp+c)D +(b\wp'+e)$ for $b=\frac{1}{2}a$ 
is discussed in \cite{U}, 
and will be discussed in a forthcoming paper 
by the first author with E. Rains in the general case.} 
\end{remark}
\end{example}

\begin{example}
Similarly, the operator (\ref{lemn}) is algebraically integrable 
when the parameter $k$ is integer and for $i=0,1$ there exist integers $m_{1i}, m_{2i}, m_{3i}, m_{4i}$ with 
$m_{1i}+m_{2i}+m_{3i}+m_{4i}=6$, which are distinct modulo 4, such that
$$
a_i=\sum_{1\le k<l\le 4} m_{ki}m_{li}-11,\quad b_i=\frac{1}{2}(\sum_{1\le k<l<j\le 4} m_{ki}m_{li}m_{ji}-6-a_i)
$$
and 
$$
c_i=m_{1i} m_{2i}m_{3i}m_{4i}.
$$

\begin{remark}
When $a_0=a_1=b_0=b_1=c_0=c_1=0$, the operator (\ref{lemn}) specializes to  
$$
L=D^4 -2k(k+1)\wp(z)D^2-2k(k+1)\wp'(z)D+k(k+1)(k+3)(k-2)\wp^2(z),
$$
(after the change of variable $z=(1+i)w$ and multiplication by $-4$), 
which is the square of the Lam\'e operator $D^2-k(k+1)\wp$
plus a constant. 
\end{remark}


\end{example}


\begin{thebibliography}{999}

\bibitem[B]{B}
G.~Bellamy: 
{\it Factorization in generalized Calogero-Moser spaces}. 
J. Algebra 321 (2009), no. 1, 338--344;
arXiv:0807.4550v1.

\bibitem[BB]{BB} 
A.~Beilinson, J.~Bernstein: 
{\it Proof of Jantzen's conjecture.}
Advances in Soviet Mathematics 16 (1993), 1--50.

\bibitem[BC]{BC} 
Y.~Berest, O.~Chalykh:
{\it Quasi-invariants of complex reflection groups}.
arXiv:0912.4518.

\bibitem[BD]{BD}
A.~Beilinson, V.~Drinfeld: 
{\it Quantization of Hitchin's integrable system and Hecke
eigensheaves}. 
preprint, available at 
http://www.math.uchicago.edu/~mitya/langlands.html 

\bibitem[BPF]{BPF} F.A ~Berezin, G.P. ~Pohil, V.M.~Finkelberg:
{\it  The Schr\"odinger equation for a system of one-dimensional particles with
point interaction.} Vestnik Moskov. Univ. Ser. I Mat. Meh.  1964  no. 1,
21--28. 

\bibitem[BE]{BE}
R.~Bezrukavnikov, P.~Etingof:
{\it Parabolic induction and restriction functors for rational Cherednik algebras}.
Selecta Math. (N.S.) 14 (2009), no. 3-4, 397--425.

\bibitem[BCS]{BCS} 
A.J. Bordner, E. Corrigan, R. Sasaki:
{\it  Generalised Calogero-Moser models and universal Lax pair operators.} 
Progr. Theoret. Phys. 102 (1999), no. 3, 499--529. 

\bibitem[BEG]{BEG} 
A.~Braverman, P.~Etingof, D.~Gaitsgory: 
{\it Quantum integrable systems and differential Galois theory}. 
Transform. Groups 2, 31--56 (1997).

\bibitem[BFV]{BFV}
V.~Buchstaber, G.~Felder, A.~Veselov:
{\it Elliptic Dunkl operators, root systems, and functional equations}.
Duke Math. J. (1994) vol. 76 (3) 885--911.

\bibitem[C1]{C1}
F.~Calogero:
{\it Solution of the one-dimensional n-body problems with quadratic and/or inversely quadratic pair potentials}. 
J. Math. Phys. 12 (1971), 419--436.

\bibitem[C2]{C2}
F.~Calogero: 
{\it Exactly solvable one-dimensional many-body problems}. 
Lett. Nouvo Cimento 13 (1975), 411--416.

\bibitem[CMR]{CMR} F. Calogero, 
C. Marchioro, O. Ragnisco: {\it Exact solution of the
classical and quantal one-dimensional many-body problems with the two-body
potential $V_a(x)=g^{2}a^{2}/{\rm sinh}^{2}(ax)$}.  Lett. Nuovo Cimento (2)
13  (1975),  no. 10, 383--387.

\bibitem[CMS]{CMS}
J. F. van Diejen and L. Vinet (Editors): {\it  Calogero-Moser-Sutherland
models.} CRM Ser. Math. Phys., Springer, New York, 2000.

\bibitem[CEO]{CEO} 
O.~Chalykh, P.~Etingof, A.~Oblomkov: 
{\it Generalized Lame operators}. 
Communications in Math. Physics. 239 (2003), no 1-2, 115--153.

\bibitem[Ch1]{Ch1} 
I.~Cherednik: {\it A unification of Knizhnik-Zamolodchikov and Dunkl operators via affine Hecke algebras}. Invent. Math. 106 (1991), pp. 411--431.

\bibitem[Ch2]{Ch2}
I.~Cherednik: 
{\it Elliptic quantum many-body problem and double affine Knizhnik-Zamolodchikov equation}. 
Comm. Math. Phys. 169 (1995), no. 2, 441--461.

\bibitem[Che]{Che} 
C.~Chevalley: 
{\it Invariants of finite groups generated by reflections}. 
Amer. J. Math. 77 (1955), 778--782.

\bibitem[CV1]{CV1} 
O.~Chalykh, A.~Veselov:
{\it Commutative rings of partial differential operators and Lie algebras}.
Comm. Math. Phys. Volume 126, Number 3 (1990), 597--611. 

\bibitem[CV2]{CV2} 
O.~Chalykh, A.~Veselov:
{\it Integrability in the theory of Schr\"oadinger operator and harmonic analysis}.
Comm. Math. Phys. 152 (1993), no. 1, 29--40. 

\bibitem[Da]{Da}
G. Darboux: {\it Sur une \'equation lin\'eare.} C. R. Acad. Sci. Paris,
t. XCIV (1882), no. 25, 145-1648.

\bibitem[D]{D}
C.~Dunkl:
{\it Differential-difference operators associated to reflection groups}. 
Trans. Amer. Math. Soc. 311 (1989), 167--183.

\bibitem[DO]{DO}
C.~Dunkl, E.~Opdam: 
{\it Dunkl operators for complex reflection groups}.
Proc. London Math. Soc. (3) 86 (2003), no. 1, 70--108.

\bibitem[E1]{E1}
P.~Etingof:
{\it Quantum integrable systems and representations of Lie algebras}.
J. Math. Phys. 36 (1995), no. 6, 2636--2651. 

\bibitem[E2]{E2}
P.~Etingof:
{\it Cherednik and Hecke algebras of varieties with a finite group action}. 
arXiv:math.QA/0406499.

\bibitem[E3]{E3} 
P.~Etingof:
{\it Calogero-Moser systems and representation theory.} 
Z\"urich Lectures in Advanced Mathematics. European Mathematical Society (EMS), 
Z\"urich, 2007; arXiv:math/0606233.

\bibitem[EG]{EG}
P.~Etingof, V.~Ginzburg:
{\it Symplectic reflection algebras, Calogero-Moser space, and deformed Harish-Chandra homomorphism}. 
Invent. Math. 147 (2002), 243--348

\bibitem[EM1]{EM1}
P.~Etingof, X.~Ma:
{\it On elliptic Dunkl operators}.
Special volume in honor of Melvin Hochster. Michigan Math. J. 57 (2008), 293--304. 

\bibitem[EM2]{EM2}
P.~Etingof, X.~Ma: 
{\it Lecture notes on Cherednik algebras}.
arXiv:1001.0432.

\bibitem[F]{F} I. Frenkel, 1976, unpublished.

\bibitem[For]{For} A.R. Forsyth: {\it Theory of Differential Equations.} Vol. 4. Dover Publications, Inc.,
New York, 1959.

\bibitem[GN]{GN} 
A.~Gorsky, N.~Nekrasov:
{\it Elliptic Calogero-Moser system from two dimensional current algebra}. 
arXiv:hep-th/9401021v1.

\bibitem[Ha]{Ha}
{\it G.H. Halphen, Trait\'e des fonctions elliptiques 
et de leurs applications}. vol. 2, Paris, 1888. 

\bibitem[He1]{He1} 
G. J. Heckman: {\it A remark on the Dunkl differential-difference
operators}. Harmonic analysis on reductive groups, Brunswick, ME,
1989; eds. W. Barker and P. Sally, Progr. Math., vol. 101,
Birkh\"auser Boston, Boston, MA, 1991, pp. 181--191.

\bibitem[He2]{He2}
G.~Heckman: 
{\it An elementary approach to the hypergeometric shift operators of Opdam}. 
Invent. Math. 103, (1991), 341--350.

\bibitem[He3]{He3}
G.J. Heckman:
{\it Root systems and hypergeometic functions, II.}
Compositio Math. 64 (1987), no. 3, 353--373.

\bibitem[HO]{HO}
G.J. Heckman, E.M. Opdam
{\it Root systems and hypergeometric funtions, I}.
Compositio Math. 64 (1987), no. 3, 329--352.

\bibitem[I]{I}
V.~Inozemtsev:
{\it Lax representation with spectral parameter on a torus for integrable particle systems}.
Lett. Math. Phys. (1989) vol. 17 (1) 11--17.

\bibitem[KKS]{KKS}
D.~Kazhdan, B.~Kostant, S.~Sternberg:
{\it Hamiltonian group actions and dynamical systems of Calogero type}. 
Comm.Pure Apll. Math., 31(1978), 481--507.

\bibitem[K1]{K1}  I.M.~Krichever:  {\it Elliptic solutions of the Kadomtsev-Petviashvili
equations and integrable systems of particles.}  Funkt. Anal. Prilozh. 14
(1980),  no. 4, 45--54.

\bibitem[K2]{K2} I.M.~Krichever: {\it Methods of algebraic geometry in the theory of
nonlinear equations.} Uspekhi Mat. Nauk  32  (1977),  no. 6(198), 183--208.

\bibitem[Ma]{Ma}
G.~Malle:
{\it Presentations for Crystallographic Complex Reflection Groups}.
Transform. Groups 1 (3), 259--277. (1996).

\bibitem[M]{M}
J.~Moser: 
{\it Three integrable Hamiltonian systems connected with isospectral deformations}.
Advances in Math. 16 (1975), 197--220.

\bibitem[OP1]{OP1}
M.~Olshanetsky, A.~Perelomov:
{\it Completely integrable hamiltonian systems connected with semisimple Lie algebras}. 
Invent. Math. 37, 93--108 (1976).

\bibitem[OP2]{OP2}
M.~Olshanetsky, A. ~Perelomov:
{\it Quantum completely integrable systems connected with semisimple Lie algebras}. 
Lett. Math. Phys. 2 (1977/78), no. 1, 7--13.

\bibitem[OP3]{OP3}
M.~Olshanetsky, A.~Perelomov:
{\it Classical integrable finite-dimensional systems related to Lie algebras}.
Phys. Rep. 71 (1981), no. 5, 313--400.

\bibitem[OP4]{OP4}
M.~Olshanetsky, A.~Perelomov:
{\em Quantum integrable systems related to Lie algebras}. 
Phys. Rep. 94, (1983), 313--404.

\bibitem[O1]{O1} E. M.~Opdam: 
{\it Root systems and hypergeometric functions III}. 
Compositio Math., 67(1988), 21-49.

\bibitem[O2]{O2} E. M.~Opdam: 
{\it Root systems and hypergeometric functions IV}.
Compositio Math., 67 (1988), 191-209.

\bibitem[Pi]{Pi} E. Picard: 
{\it Sur les \'equations diff\'erentielles lin\'eaires \`a
coefficients doublement p\'eriodiques.}
Journal f\"ur die reine und angewandte
Mathematik, v. 90, pp. 281 - 302.

\bibitem[Po]{Po} 
V.~Popov:
Discrete Complex Reflection Groups (Rijksuniv. Utrecht, Utrecht, 1982).
Commun. Math. Inst., Rijksuniv. Utrecht 15. 

\bibitem[S]{S}
B.~Sutherland:
{\it Exact results for a quantum many-body problem in one-dimension II}. 
Phys. Rev. A5 (1972), 1372--1376.

\bibitem[U]{U} K. Unterkofler: {\it On the solutions of Halphen's equation}. Differential and Integral Equations, 14 (2001), 1025-1050. 

\bibitem[Ve]{Ve} A.P. Veselov: {\it Integrable maps.} Russian Math. Surveys, 46:5 (1991), 3­45.

\bibitem[VSC]{VSC} A.P.~Veselov, K.L.~Styrkas, O.A.~Chalykh:
{\it Algebraic integrability for the Schr\"odinger equation, and groups generated
by reflections.} Theoret. and Math. Phys. 94 (1993), no. 2, 182--197.

\bibitem[WW]{WW} E.T. Whittaker, G.N. Watson: {\it A Course of Modern Analysis.}
Cambridge Univ. Press, 1963.

\end{thebibliography}
\end{document}